\documentclass[12pt,reqno]{amsart}
\usepackage{geometry}                
\usepackage{subfig}
\usepackage{amsmath,amssymb}
\usepackage{graphicx}
\usepackage{amssymb,latexsym, amsmath}
\usepackage{amsfonts,amsthm}
\usepackage{amscd}
\usepackage{mathrsfs}
\usepackage{hyperref}
\usepackage{eucal}
\usepackage{multicol}
\usepackage{epstopdf}
\usepackage{amsgen}
\usepackage{xspace}
\usepackage{verbatim}
\usepackage{stmaryrd}
\usepackage{graphicx}
\usepackage{caption}

\newcommand{\ap}{\alpha}
\newcommand{\ga}{\gamma}
\newcommand{\de}{\delta}
\newcommand{\gb}{\beta}
\newcommand{\G}{\Gamma}
\newcommand{\gl}{\lambda}

\newcommand{\ve}{\varepsilon}
\newcommand{\dca}{D_c^\alpha}

\newcommand{\beq}{\begin{equation}}
\newcommand{\eeq}{\end{equation}}
\newcommand{\bea}{\begin{align}}
\newcommand{\eea}{\end{align}}
\newcommand{\bthm}{\begin{theorem}}
\newcommand{\ethm}{\end{theorem}}
\newcommand{\bpr}{\begin{proof}}
\newcommand{\epr}{\end{proof}}
\newcommand{\bcl}{\begin{corollary}}
\newcommand{\ecl}{\end{corollary}}
\newcommand{\bpn}{\begin{proposition}}
\newcommand{\epn}{\end{proposition}}
\newcommand{\bre}{\begin{remark}}
\newcommand{\ere}{\end{remark}}
\newcommand{\bdf}{\begin{definition}}
\newcommand{\edf}{\end{definition}}
\newcommand{\bss}{\begin{align*}}
\newcommand{\ess}{\end{align*}}

\newcommand{\bl}{\label}

\newtheorem{theorem}{Theorem}[section]
\newtheorem{corollary}[theorem]{Corollary}

\newtheorem{lemma}[theorem]{Lemma}
\newtheorem{proposition}[theorem]{Proposition}

\theoremstyle{definition}
\newtheorem{definition}[theorem]{Definition}
\theoremstyle{remark}
\newtheorem{remark}{Remark}

\numberwithin{equation}{section}

\begin{document}

\title[Synchronization of Time-Fractional Hopfield Neural Networks]{Robust Synchronization of Time-Fractional Memristive Hopfield Neural Networks}

\author[Y. You]{Yuncheng You}
\address{University of South Florida, Tampa, FL 33620, USA}
\email{yygwmp@gmail.com}

\thanks{}


\date{\today}



\begin{abstract} 
In this paper we study robust synchronization of time-fractional Hopfield neural networks with memristive couplings and Hebbian learning rules. This new model of artificial neural networks exhibits strong memory and long-range path-dependence in learning processes. Through scaled group estimates it is proved that under rather general assumptions the solution dynamics is globally dissipative. The main result established a threshold condition for achieving robust synchronization of the neural networks if it is satisfied by the interneuron coupling strength coefficient. The synchronizing threshold is explicitly computable in terms of the original parameters and strictly decreasing for the fractional order $\ap \in (0, 1)$. 
\end{abstract}

\maketitle
 
\section{\textbf{Introduction}}

Dynamical analysis of artificial neural networks plays an essential role in deep learning and all-around AI applications. In recent two decades many biological and artificial neural network models of differential equations have been proposed and studied on topics of stability and pattern formation mainly by approaches of matrix spectrum estimation and computational simulations. As a type of recurrent neural networks, Hopfield neural networks \cite{Hp, HJ, OC} with memristive synapses \cite{Chua, V} are extensively used for the preeminent performances in many frontier research applications \cite{Ml, Li, E, D}. 

For biological neural network models in terms of diffusive Hindmarsh-Rose equations \cite{Y2}, diffusive FitzHugh-Nagumo equations \cite{YTT}, and general reaction-diffusion equations \cite{YT} with memristive synapses, the exponential synchronization results have been proved by the author's group. For the memristive Hopfield neural networks, very recently it was shown \cite{Y} that only approximate synchronization can be achieved instead of complete synchronization, due to mismatched weight dynamics and heterogeneous activation functions among the neurons.

In recent two decades, along with fractional-order derivatives and related calculus, researches on time-fractional differential equations have been booming in modeling dynamics of physical and interdisciplinary processes as well as Hopfield neural networks \cite{BM, KS}, ranging from viscoelastic liquids \cite{BC}, electrochemical processes \cite{HL}, dielectric polarization \cite{RM}, image processing and encryption \cite{JV}, mathematical finance \cite{SM, AE}, to machine learning \cite{RT}. The fractional differential equation models have the advantage of continuum long-term memory properties in comparison with classical integer-order differential equations with discrete time-delay terms. 

From the biosimulation point of view, fractional-order formulation of artificial neural networks can be analogously justified by some reported researches on biological neural systems, such as the oculomotor integrator \cite{TA}, which is a neural network in the brainstem that quickly converts a temporary eye velocity signal to a persistent eye position command of fractional order less than one. Another evidence is that fractional-order differential model can be used to efficiently emulate stimulus anticipation and oscillatory neuron firing dynamics \cite{BL}, which is a cognitive preparatory measure triggered by the emotional potentials called SPN (stimulus preceding negativity). In the sense fractional model dynamics is an important aspect for the frontier BMI (brain-machine interface) technology. 

Synchronization dynamics of fractional-order memristive Hopfield neural networks is an important topic, which is more extensive and complicated than the classical stability properties converging limited to equilibrium states \cite{PZ, WL}, because it also includes the chaotic synchronization and the approximate synchronization.

In this paper, we propose a new model of Caputo type time-fractional memristive Hopfield neural networks with Hebbian learning rule. It can also be called fractional memristive Hopfield-Hebbian neural networks (briefly fmHHNN). We shall investigate robust synchronization of the solution trajectories for his model.   

Consider a fractional-order memristive Hopfield neural network composed of $m$ neuron nodes and denoted by $\mathcal{NW} = \{\mathcal{N}_i : i = 1, 2, \cdots, m\}$. The state variable $u_i(t)$ for each neuron $\mathcal{N}_i, 1 \leq i \leq m$, the symmetric weights $w_{ij}(t)$ updated by the Hebbian learning rule \cite{H, VB} before synaptic activation functions $f_j(u_j)$, and the memristor's window function $\rho(t)$ are governed by fractional differential equations of the same fractional order $\ap$ in this model:
\beq \bl{Meq}
	\begin{split}
	&D_c^\ap u_i (t) = - a_i u_i + \sum_{j =1}^m w_{ij} \, f_j (u_j) + k\, \psi_i (\rho) u_i + J_i - P \sum_{j=1}^m (u_i - u_j), \;\; 1 \leq i \leq m,  \\
	&D_c^\ap w_{ij} (t) = - c_{ij}\, w_{ij} + \gl_{ij} \, f_i (u_i) f_j (u_j), \quad  1 \leq i, j \leq m,   \\
	&D_c^\ap \rho(t) = \sum_{i=1}^m \ga_i u_i - b \rho,  \quad  \text{where} \quad  t > 0, \quad \ap \in (0, 1).
	\end{split} 
\eeq
In this model, the Caputo type fractional time derivative of order $\ap \in (0, 1)$ denoted by $\dca$ of a scalar function $y(t)$ is defined by 
$$
	\dca y(t) = \frac{1}{\G (1 - \ap)} \int_0^t (t - s)^{-\ap}\, y^\prime (s)\, ds, \quad  t > 0,
$$
and $\G(z) = \int_0^\infty t^{z-1} e^{-t}\, dt, \,z > 0,$ is the Gamma function, for $y(\cdot) \in C[0, \infty)$ and $y^\prime (\cdot) \in L^1_{loc} [0, \infty)$. The coefficient $P > 0$ in \eqref{Meq} is called the interneuron coupling strength. Note that the state variable $u_i(t)$ simulates the electric potential of a biological neuron. The initial states of the system \eqref{Meq} will be denoted by 
\beq \bl{inc}
	 u_i^0 = u_i(0),  \quad w_{ij}^0 = w_{ij}(0),  \quad \rho^0 = \rho (0), \quad 1 \leq i, j \leq m.
\eeq
Here $w^0_{ij} = w^0_{ji}$ for any $i, j$. The memristive feature of this model is reflected by the memductance-potential synapse term $k \psi_i (\rho) u_i$ for each neuron $\mathcal{N}_i$. The Hebbian learning rule used in \eqref{Meq} for synaptic weights  $w_{ij} (t)$ demonstrated good performance in unsupervised machine learning tasks on certain datasets like MNIST datasets and CIFAR datasets in image recognition and classification \cite{VFR}. 

We assume that the scalar functions $f_i (s), 1 \leq i \leq m$ are locally Lipschitz continuous functions and the memristor window's functions are $\psi_i (s) = s(\eta_i - s), 1 \leq i \leq m$. The parameters $a_i, b, k, \eta_i, c_{ij}$ can be any given positive constants, while the input/bias constants $J_i$ and coefficients $\gl_{ij}, \ga_i$ can be any given real numbers. More specific assumptions on parameters and element functions will be made in Section 3.

In this paper we shall define robust synchronization for artificial neural network models and investigate this new model \eqref{Meq} of fractional-order memristive Hopfield neural networks to rigorously prove that robust synchronization of this fmHHNN will be triggered if the network coupling strength coefficient $P$ satisfies a computable threshold condition. 

The rest of the paper is organized as follows. In Section 2 we recall preliminary basics of fractional calculus and time-fractional differential equations. In particular we shall prove three useful tool Lemmas, especially a new form of fractional Gronwall inequality. A definition of robust synchronization will then be introduced. In Section 3 we show the dissipative dynamics of global solutions to the initial value problem \eqref{Meq}-\eqref{inc} in terms of the existence of an absorbing set. In Section 4 we shall prove the main result on robust synchronization of the fractional memristive Hopfield-Hebbian neural networks by the approach of direct uniform estimates of the interneuron differencing equations. Conclusions will be in the end.

\section{\textbf{Fractional Analysis Lemmas}}

Fractional-order derivatives and fractional calculus broadly expanded the mathematical fields in modeling physical processes and phenomenological dynamics by fractional differential equations (FDE). The framework of FDE provides a flexible incorporation of longterm memory of continuum past states in the trajectories of real time evolution. The basic theory and some applications of fractional calculus and FDE can be found in books \cite{KST, KD, BJ}.

Caputo fractional derivative $\dca y(t)$ defined in Section 1 is strictly called the left-sided Djrbashian-Caputo fractional derivative \cite{BJ, Ca}, usually for absolutely continuous functions $y \in AC[0, T)$ which annihilated the singularity at the initial point $t = 0$. Compare with Riemann-Liouville fractional derivative and Gr\"{u}nwald-Letnikov fractional derivative, Caputo fractional derivative can easily handle the initial conditions in a way that matches the physically observable ordinary derivatives at $t = 0$, making it most suitable and popular for modeling initial value problems of evolutionary equations in real world applications. But the product rule and the chain rule for the ordinary derivatives are not valid for Caputo fractional derivatives in analysis.

First we present a useful Lemma \cite{AA} with a brief proof, which will be used in conducting \emph{a priori} estimates of solutions for FDE. 
\begin{lemma} \bl{L1}
       For $\ap \in (0, 1)$ and any given  $T > 0$, if $f \in AC [0, T]$, then it holds that
\beq \bl{prl}
	 \frac{1}{2}\, \dca f^2(t) \leq f(t) \dca f(t), \quad \text{for} \;\; t \in (0,T).
\eeq
\end{lemma}
\begin{proof}
By change order of integration and then integration by parts, we have
\begin{equation*}
	\begin{split}
	&\G(1 - \ap) \left[f(t) \dca f(t) - \frac{1}{2}\, \dca f^2(t) \right]  \\
	= &\, f(t) \int_0^t (t - s)^{-\ap} f^\prime (s)\, ds - \int_0^t (t - s)^{-\ap} f(s) f^\prime(s)\,ds  \\
	= &\,\int_0^t (t - s)^{-\ap} f^\prime (s)\, [ f(t)- f(s) ]\, ds = \int_0^t (t - s)^\ap f^\prime (s) \int_s^t f^\prime (x)\, dx\,ds  \\
	\overset{\text{CoI}}{=} &\, \int_0^t f^\prime (x) \int_0^x (t - s)^\ap f^\prime (s)\, ds\,dx = \int_0^t (t - x)^\ap \frac{f^\prime (x)}{(t - x)^\ap} \int_0^x (t - s)^\ap f^\prime (s)\, ds\,dx  \\
	= &\, \frac{1}{2} \int_0^t (t - x)^\ap \frac{d}{dx} \left(\int_0^x (t - s)^{-\ap} f^\prime(s)\, ds \right)^2 dx  \\
	\overset{\text{IbP}}{=} &\, \frac{\ap}{2} \int_0^t (t - x)^{\ap -1} \left(\int_0^x (t - s)^{-\ap} f^\prime(s)\, ds \right)^2 dx \geq 0, \quad  t \in (0, T).
	\end{split}
\end{equation*}
Since $\G(1 - \ap) > 0$ for $\ap \in (0, 1)$, the inequality \eqref{prl} is proved.
\end{proof}

The Mittag-Leffler function $E_\ap (z)$ and the two-parameter Mittag-Leffler function $E_{\ap_1, \ap_2} (z)$ \cite{ML, GK} defined by  
$$
	E_\ap (z) = \sum_{n = 0}^\infty \frac{z^n}{\G(n \ap + 1)},  \quad  E_{\ap_1, \ap_2} (z) = \sum_{n = 0}^\infty \frac{z^n}{\G(n \ap_1 + \ap_2)},  \quad z \in \mathbb{C},
$$
have an essential role to play in fractional differential equations and analysis. $E_\ap (z)$ is an entire function and $E_1(z) = e^z$. The function $E_\ap (-x)$ with real variable $x > 0$ is especially more involved in treatment of FDE. It is completely monotone shown in \cite[Corollary 3.2]{BJ} meaning the derivatives $(-1)^n E_\ap^{(n)} (-x) \geq 0$ for $x > 0$ and all positive integer orders $n$. 

The following Lemma provides a very useful tool in study of dynamics with respect to global solutions of fractional-order differential equations. 

\begin{lemma} \textup{(Fractional Gronwall Inequality)}  \bl{L2}
      If a nonnegative function $x(t)$ is absolutely continuous for $t \geq 0$ and satisfies the inequality
\beq \bl{pq}
	\dca x(t) \leq p - q\, x(t), \quad  \text{for} \;\;  t > 0,
\eeq
where the fractional order is $\ap \in (0, 1)$, $p$ and $q$ are positive constants, then it holds that
\beq \bl{fG}
	\begin{split}
	x(t) &\,\leq x(0)\, E_\ap (-q\, t^\ap) + \, \frac{p}{q}\,\G(\ap) \left(1 - E_\ap (- q\, t^\ap) \right)      \\
	&\, < x(0) \left( \frac{\G(1 + \ap)}{\G(1 + \ap) + q t^\ap} \right) + \, \frac{p}{q}\,\G(\ap) ,  \quad \text{for} \; \; t > 0.
	\end{split}
\eeq
\end{lemma}

\begin{proof}
Parallel to the variation of constants formula for linear ODEs, it can be shown \cite[Proposition 4.5]{BJ} that $x(t)$ in the inequality \eqref{pq} satisfies 
\beq \bl{cv}
	\begin{split}
	x(t) &\,= x(0) E_\ap (- q\, t^\ap) + p \int_o^t \frac{1}{(t - s)^{1 - \ap}} E_{\ap, \ap} (-q (t - s)^\ap)\, ds  \\
	&\, - \int_0^t \frac{R(t)}{(t - s)^{1 - \ap}} E_{\ap, \ap} (-q (t - s)^\ap)\, ds \\
	&\, \leq x(0) E_\ap (- q\, t^\ap) + p \int_o^t \frac{1}{(t - s)^{1 - \ap}} E_{\ap, \ap} (-q (t - s)^\ap)\, ds,  \quad  t > 0, 
	\end{split}
\eeq
where $- R(t)$ is the nonpositive gap function between the two sides of the inequality \eqref{pq} and for $0 < \ap < 1$ the function $E_{\ap, \ap} (- x)$ is always positive on $x > 0$ due to the complete monotonicity as mentioned. 

By the integral relation with the two-parameter Mittag-Leffler functions \cite[Section 4.4]{GK}, we have 
\beq \bl{nt}
	\frac{1}{\G (\ap)} \int_o^t \frac{1}{(t - s)^{1 - \ap}} E_{\ap, \ap} (-q (t - s)^\ap)\, ds = t^\ap E_{\ap, \ap +1} (- q\, t^\ap).
\eeq
From \eqref{cv} and \eqref{nt} it follows that
\beq \bl{xp}
      x(t) \leq x(0) E_\ap (- q\, t^\ap) + p\,\G(\ap)\, t^\ap E_{\ap, \ap +1} (- q\, t^\ap), \quad t > 0.
\eeq
Note that \cite[Theorem 4.2]{KD}
$$
	E_{\ap, \gb}\, (x) = x\, E_{\ap, \ap + \gb}\, (x) + \frac{1}{\G(\gb)}.
$$
Then we have
\beq \bl{ap}
	 t^\ap E_{\ap, \ap +1} (- q\, t^\ap) = \frac{1}{q} \left[\frac{1}{\G(1)} - E_{\ap, 1}\,(- q\,t^\ap) \right] = \frac{1}{q} \left[1 - E_{\ap}\,(- q\,t^\ap)\right].
\eeq
Moreover, Theorem 3.6 in \cite{BJ} shows that the Mittag-Leffler function $E_\ap (- x)$ admits the two-side bounds
\beq \bl{Eb}
	\frac{1}{1 + \G(1 - \ap) x} \leq E_\ap (- x) \leq \frac{\G(1 + \ap)}{\G(1 + \ap) + x} , \quad  x > 0.
\eeq
Finally substitute \eqref{ap} and then \eqref{Eb} with $x = q\, t^\ap$ into \eqref{xp}, we end up with 
\begin{equation*}
	\begin{split}
	x(t) &\,\leq x(0) E_\ap (- q\, t^\ap) + p\,\G(\ap)\, t^\ap E_{\ap, \ap +1} (- q\, t^\ap)  \\[4pt]
	&\, = x(0) E_\ap (- q\, t^\ap) + \frac{p}{q}\, \G(\ap) \left(1 - E_{\ap}\,(- q\,t^\ap)\right)   \\
	&\, \leq x(0) \left(\frac{\G(1 + \ap)}{\G(1 + \ap) + q\, t^\ap} \right) + \frac{p}{q}\, \G(\ap)\left(\frac{\G(1 - \ap) q\,t^\ap}{1 + \G(1 - \ap) q\,t^\ap}\right)     \\
	&\, <  x(0) \left(\frac{\G(1 + \ap)}{\G(1 + \ap) + q\, t^\ap} \right) + \frac{p}{q}\, \G(\ap)\, ,  \quad  t > 0.
	\end{split}
\end{equation*}
The fractional Gronwall inequality \eqref{fG} is proved.
\end{proof}

In regard to fractional-order differential equations, the existence of global solutions for an initial value problem needs to be carefully addressed. Continuation of a local solution in time is a subtle issue because two local solution segments in time may not be simply concatenated like ODEs. In fact  two different solution trajectories of a Caputo FDE can intersect \cite[Theorem 6.1]{CT}. 

The next lemma presents an implementable sufficient condition, other than the stringent global Lipschitz condition, for the existence of global solutions in time with respect to initial value problems (IVP) of autonomous fractional differential equations.

\begin{lemma} \bl {L3}
	For an initial value problem of autonomous Caputo fractional differential equation\textup{(}s\textup{)}
$$
	\dca x(t) = f(x), \quad x(0) = x_0 \in \Omega \subset \mathbb{R}^n,
$$
if the scalar or vector function $f(x)$ is locally Lipschitz on its bounded or unbounded domain $\Omega \subset \mathbb{R}^n$ and a solution of this IVP satisfies 
\beq \bl{dp}
	\|x(t)\| \leq \|x_0\|\, C(t) + Q, \qquad  t \in I_{max} = [0, T_{max}),
\eeq
where $C(t)$ is a positive non-increasing continuous function, $C(t) \to 0$ as $t \to \infty$, and $Q > 0$ is a constant, then the maximal existence interval of this solution $I_{max} = [0, \infty)$ so that $x(t)$ is a global solution in time.  
\end{lemma}

\begin{proof}
	According to the dissipative condition \eqref{dp}, any solution trajectory satisfies $\{x(t): t \in I_{max}\} \subset B_{x_0} = \{x \in \Omega: \|x\| \leq \|x_0\| C(0) + Q + 1\}$, which is a bounded ball in the state space $\mathbb{R}^n$. By the locally Lipschitz continuous condition, the vector field $f(x)$ is globally Lipschitz within the bounded set $B_{x_0}$. Then we can apply the standard existence and uniqueness theorems \cite{KD, BJ, LS} to assert that this solution $x(t)$ is a unique global solution, $I_{max} = [0, \infty)$, of this initial value problem.
\end{proof}

Absorbing set \cite{GY} defined below is an important concept in study of dissipative evolutionary dynamics.

\begin{definition} \bl{Ab}
	For a $d$-dimensional system $S$ of time-fractional differential equations, a bounded subset $B^* \subset \mathbb{R}^d$ is called an absorbing set if for any given bounded set $B$ in the state space $\mathbb{R}^d$ there is a finite time $T_B \geq 0$ such that all the solution trajectories of this system $S$ initially started in the set $B$ will permanently enter the set $B^*$ when $t  > T_B$. A system of fractional differential equations is called \emph{dissipative} if there exists an absorbing set. 
\end{definition}

The Young's inequality is widely seen in management of differential inequalities.  It states that for any positive numbers $x$ and $y$, if $\frac{1}{p} + \frac{1}{q} = 1$ and $p > 1, q > 1$ are constants, then one has
\beq \bl{Yg}
	x\,y \leq \frac{1}{p}\, \ve \,x^p + \frac{1}{q}\, C(\ve, p)\, y^q \leq \ve \,x^p + C(\ve, p)\, y^q, 
\eeq
where $C(\ve, p) = \ve^{-q/p}$ and the constant $\ve > 0$ can be arbitrarily given. 

Finally in this section, we introduce the definition of robust synchronization.
\begin{definition} \bl{Df}
	For an artificial neural network model $\mathcal{NW}$ in terms of time-fractional differential equations such as \eqref{Meq}, define its synchronous degree to be
\beq \bl{deg}
	\text{deg}_s (\mathcal{NW}) = \sup_{u^0 \in \,\Omega \,\subset \,\mathbb{R}^d} \left\{\max_{1 \leq i < j \leq m} \left\{\limsup_{t \to \infty} |u_i(t) - u_j(t)| \right\} \right\}.
\eeq
where $u^0 = (u_1 (0), \cdots, u_d (0))$ is any initial state of the neurons in the domain $\Omega$ of the state space $\mathbb{R}^d$. The neural network $\mathcal{NW}$ is called robust synchronizable if for any small $\ve > 0$, there exists an interneuron coupling threshold $P^* (\ve) > 0$ such that $\text{deg}_s (\mathcal{NW}) < \ve$ for all $P > P^*(\ve)$. 
\end{definition}

\section{\textbf{Dissipative Dynamics of Fractional Hopfield Neural Networks}}

The mathematical model \eqref{Meq}-\eqref{inc} of the Caputo fractional neural network $\mathcal{NW}$ with Hebbian learning rule can be formulated into an initial value problem of the evolutionary equation:
\begin{equation} \label{pb}
\begin{split}
	&\dca \,g(t) = F(g),   \quad \:\;  t > 0,  \\[3pt]
	&g(0) = g^0 \in X = \mathbb{R}^{m(1 + m) +1}.
\end{split}
\end{equation}
In \eqref{pb} the column vector function $g(t) = \text{col} \left(u_i (t), w_{ij}(t), \rho(t)): 1 \leq i, j \leq m\right)$ and the initial state vector $g(0) = g^0 = \text{col} ((u_i^0, \, w_{ij}^0, \rho^0): 1 \leq i, j \leq m)$.

In \eqref{pb} the $m(1 + m) +1$-dimensional nonlinear vector function
\begin{equation} \label{opf}
F(g) =
\begin{pmatrix}
	- a_i u_i + \sum_{j =1}^m w_{ij} f_j (u_j) + k\, \psi_i (\rho) u_i + J_1 - P \, \sum_{j=1}^m (u_i - u_j) : 1 \leq i \leq m \\[10pt]
	- c_{ij} w_{ij} + \gl_{ij} f_i(u_i) f_j (u_j) : 1 \leq i, j \leq m  \\[10pt]
	- b\, \rho + \sum_{i=1}^m \ga_i u_i
\end{pmatrix}
\end{equation}
is locally Lipschitz continuous. \eqref{pb} is the concise version of the model \eqref{Meq}-\eqref{inc} of the fractional memristive Hopfield-Hebbian neural networks.

Here are some remarks on the versatility and features: [1] The activation functions $\{f_i(x): 1 \leq i \leq m\}$ can be heterogeneous for different neurons, including Softplus $\textup{Sp} (x) = \ln \, (1 + e^x), \,\textup{ReLU}(x)$, $\textup{GeLU} (x) = (x/2)[1 + \textup{Erf}\, (x/\sqrt{2})]$, which are plausibly assumed to be uniformly bounded functions since the analogous potential of biological neurons are bounded from the resting level up to the bursting limit. [2] This model is compatible with typical memristor (chips or circuits) window functions such as sigmoidal function $\tanh (s)$, linear function $1 - \eta | s |$, quadratic function $1 - \eta s^2$, Jogelker memristor $1 - (2s -1)^{2q}$, and Strukov-Williams memristor $s (\eta - s)$ which is taken in this model and used in CMOS technology.  [3] The synaptic weights $w_{ij}(t)$ updated real time by the Hebbian rule actually incarnate the key biological neuron factors of locality, synaptic dissipativity, cooperative ($\gl_{ij} > 0$) or competitive ($\gl_{ij} < 0$) attributes. The matrix of Hebbian coefficients $\{\gl_{ij}\}$ describes topological pathways of information through the network and can be sparse in training or learning networks. 

We make the following Assumption and set up new notations:
\beq \bl{Asp}
	\begin{split}
	&a = \min\, \{a_i: 1 \leq i \leq m\}  > \frac{1}{2}\, k \eta^2, \;\;  |f_i(s)| \leq \gb, \;\;  J = \max\, \{|J_i|: 1 \leq i \leq m\},    \\
        &\psi_i (s) = s(\eta_i - s), \; \eta = \max\, \{\eta_i: 1 \leq i \leq m\}, \; \ga = \max\, \{ |\ga_i|: 1 \leq i \leq m\},   \\[2pt]
        &\gl = \max\, \{|\gl_{ij}|: 1 \leq i, j \leq m\}, \quad  c = \min\,\{c_{ij}: 1 \leq i, j \leq m\} > 0. 
	\end{split}
\eeq
First we address the existence of global solutions for this initial value problem \eqref{pb}. 

\begin{theorem} \label{T1}
	Given any initial state $g^0 \in X$, there exists a unique global solution $g(t; g^0),\;  t \in [0, \infty),$ for the initial value problem \eqref{pb} of the fractional memristive Hopfield neural network $\mathcal{NW}$ described by the model \eqref{Meq} with the Assumption \eqref{Asp}. 
\end{theorem}

\begin{proof} 
By the standard theorems on solutions of Caputo fractional differential equations \cite{KD, BJ} and the locally Lipschitz continuous condition here, for any given initial state $g^0 \in X$ there exists a unique local solution $g(t, g^0)$ of \eqref{pb}, whose maximum existence interval is $I_{max} = [0, T_{max}(g^0))$ may depending on $g^0$. 
	
Multiply the $w_{ij}$-equation in \eqref{Meq} by $w_{ij}(t)$. Lemma \ref{L1} with \eqref{Asp} shows that
\beq \bl{wd}
	\begin{split}
	\frac{1}{2} \dca w^2_{ij}(t) & \leq w_{ij}(t) \dca w_{ij}(t) = - c_{ij}\, w^2_{ij}(t) + \gl_{ij} \,w_{ij}(t) f_i(u_i(t)) f_j(u_j(t))   \\
	& \leq - c \,w^2_{ij}(t) + \gl \,\gb^2\, w_{ij}(t) \leq - \frac{1}{2}\, c\, w^2_{ij}(t) + \frac{1}{2 c}\, \gl^2 \gb^4, \quad  t \in I_{max}.
	\end{split}
\eeq
Then by Lemma \ref{L2} we have
\beq \bl{wij}
	w^2_{ij}(t) \leq w^2_{ij}(0)\left( \frac{\G(1 + \ap)}{\G(1 + \ap) + c\, t^\ap} \right) + \frac{\gl^2 \gb^4}{c^2}\G(\ap) \leq w^2_{ij}(0) + \frac{\gl^2 \gb^4}{c^2} \G(\ap),  \;\; 1 \leq i, j \leq m.
\eeq
Denote by $W_0 = [\max \{\sum_{j=1}^m w^2_{ij}(0): i = 1, \cdots, m\}]^{1/2}$. It follows that
\beq \bl{wb}
	\sum_{i=1}^m \sum_{j=1}^m w^2_{ij}(t) \leq m W_0^2 \left( \frac{\G(1 + \ap)}{\G(1 + \ap) + c\, t^\ap} \right) + \frac{m^2\gl^2 \gb^4}{c^2} \G(\ap),  \;\; t \in I_{max}.
\eeq

Next multiply the $u_i$-equation in \eqref{Meq} by $C_1 u_i(t)$ for $1 \leq i \leq m$ and sum them up. Using \eqref{Asp} and \eqref{wb} we get	 
\beq \bl{ut}
	\begin{split}
	& \frac{1}{2} \dca \sum_{i = 1}^m C_1 u^2_i + \frac{1}{2}C_1 P \,\sum_{i=1}^m \sum_{j=1}^m \, (u_i - u_j)^2 \leq C_1 u_i \dca u_i + \frac{1}{2}C_1 P \,\sum_{i=1}^m \sum_{j=1}^m \, (u_i - u_j)^2  \\
	= &\,\sum_{i=1}^m C_1 \left[- a_i u^2_i  + \sum_{j =1}^m w_{ij}(t) f_j (u_j) u_i + k \psi_i (\rho) u_i^2 +  J_i u_i \right]   \\
	\leq &\,\sum_{i=1}^m C_1 \left[- a u_i^2 + \sqrt{\sum_{j=1}^m w_{ij}^2(0) G_\ap (t) + \frac{m \gl^2 \gb^4}{c^2} \G(\ap)}\; \gb |u_i | + \left[k \rho (\eta_i - \rho ) u_i^2 + J |u_i| \right] \right]   \\
	\leq &\, - \sum_{i=1}^m C_1 \left[a - \frac{1}{2} k \eta^2\right] u_i^2 (t) + \sum_{i=1}^m C_1 \left[J + \gb \left[W_0 \sqrt{G_\ap(t)} + \frac{1}{c}\sqrt{m \G(\ap)} \gl \gb^2 \right] \right] |u_i (t) |   \\
	\leq &\, - \frac{1}{2} \sum_{i=1}^m C_1\left[a - \frac{1}{2} k \eta^2\right] u_i^2(t) + \frac{1}{2} \,C_1 m  \frac{\left[J +  \gb \left[W_0 \sqrt{G_\ap (t)} + \frac{1}{c}\sqrt{m \G(\ap)} \gl \gb^2 \right]\,\right]^2}{a - \frac{1}{2} k \eta^2} 
	\end{split}
\eeq
for $t \in I_{max}$, where  $C_1 > 0$ is a scaling constant yet to be specified, the double sum of interneuron coupling terms in the first inequality of \eqref{ut} comes from  
\begin{gather*}
	- C_1 P \sum_{i=1}^m \sum_{j=1}^m \left[ (u_j - u_i)u_i + (u_i - u_j)u_j\right] \\
	= - C_1 P \sum_{1 \leq i < j \leq m} (u_i - u_j)^2 = - \frac{1}{2} C_1 P \,\sum_{i=1}^m \sum_{j=1}^m \, (u_i - u_j)^2    
\end{gather*}
and moved to the left side with the derivative terms. In the third inequality of \eqref{ut}, 
\beq \bl{Gt}
	G_\ap (t) =  \frac{\G(1 + \ap)}{\G(1 + \ap) + c\, t^\ap} \leq 1 \quad \text{and} \quad \lim_{t \to \infty} G_\ap (t) = 0.
\eeq 
We also treated the memductance synaptic term in the third inequality of \eqref{ut} as 
$$
	k \rho\, (\eta_i - \rho ) u_i^2 = k (\rho u_i) (\eta u_i) - k \rho^2 u_i^2 \leq \frac{1}{2} k \left(\eta^2 u_i^2 + \rho^2 u_i^2 \right) - k \rho^2 u_i^2 \leq \frac{1}{2} k\, \eta^2 u_i^2.
$$
Cauchy inequality is used for the second sum to reach the last inequality of \eqref{ut}.

Then multiply the $\rho$-equation in \eqref{Meq} with $\rho (t)$. Using Young's inequality \eqref{Yg} we have
\beq \bl{rt}
	\begin{split}
	&\frac{1}{2} \dca \rho^2 (t) \leq \rho(t) \dca \rho (t) = \sum_{i=1}^m \ga_i \,u_i(t) \rho(t) - b\, \rho^2(t)  \\[2pt]
	\leq \sum_{i=1}^m &\, \left[\frac{m \ga_i^2}{2b} u_i^2(t) + \frac{b}{2m} \rho^2(t) \right] - b \rho^2(t) = \sum_{i=1}^m \frac{m \ga_i^2}{2b} \,u_i^2 (t) - \frac{b}{2} \,\rho^2(t), \;\;  t \in I_{max}.
	\end{split}
\eeq
Combine the two differential inequalities \eqref{ut} and \eqref{rt} together. It yields 
\beq \bl{cur}
	\begin{split} 
	&\dca \left(\sum_{i = 1}^m C_1 u^2_i + \rho^2 (t) \right) + C_1 P \,\sum_{i=1}^m \sum_{j=1}^m \, (u_i - u_j)^2     \\
	\leq & - \sum_{i=1}^m C_1\left[a - \frac{1}{2} k \eta^2\right] u_i^2(t) + C_1 m \frac{\left[J +  \gb \left[W_0 \sqrt{G_\ap (t)} + \frac{1}{c}\sqrt{m \G(\ap)} \gl \gb^2 \right]\,\right]^2}{a - \frac{1}{2} k \eta^2}   \\
	& + \left(\sum_{i=1}^m \frac{m \ga^2}{b} \,u_i^2 (t) - b\,\rho^2(t) \right),  \quad t \in I_{max}.
	\end{split}
\eeq
Now we choose $C_1$ to be the positive constant 
\beq \bl{C1}
	C_1 = \frac{\frac{ m\ga^2}{b} + b}{a - \frac{1}{2} k \eta^2}  \; \quad \text{which means} \quad C_1\left(a - \frac{1}{2} k \eta^2\right) - \frac{m \ga^2}{b} = b\, .
\eeq
Substitute \eqref{C1} in \eqref{cur}. Since $C_1 P \,\sum_{i=1}^m \sum_{j=1}^m \, (u_i - u_j)^2 \geq 0$ and $G_\ap(t) \leq 1$, it results in
\beq \bl{UR}
	\dca \left(\sum_{i = 1}^m C_1 u_i^2(t) + \rho^2(t)\right) + b \left(\sum_{i=1}^m u_i^2(t) + \rho^2(t) \right) \leq C_2 (W_0), \quad  t \in I_{max},
\eeq
where 
\beq \bl{C2}
        C_2 (W_0) = m \left(\frac{m \ga^2}{b} + b \right) \left[J +  \gb \left[W_0 + \frac{1}{c}\sqrt{m \G(\ap)} \gl \gb^2 \right]\,\right]^2 \left(a - \frac{1}{2} k \eta^2\right)^{- 2}	
\eeq
is a uniform constant independent of any initial state $(u_1^0, \cdots, u_m^0, \rho^0)$ of the network neurons. The fractional differential inequality \eqref{UR} infers that
\beq \bl{urb}
	\dca \left(\sum_{i = 1}^m C_1 u_i^2(t) + \rho^2(t)\right) + \de \left(\sum_{i=1}^m C_1\, u_i^2(t) + \rho^2(t) \right) \leq C_2 (W_0), \;\;  t \in I_{max},
\eeq
where 
\beq \bl{dr}
	\de = b \, \min \left\{\frac{1}{C_1},\, 1\right\}.
\eeq
Apply the fractional Gronwall inequality in Lemma \ref{L2} to \eqref{urb}, It is shown that
\beq \bl{Gur}
	\begin{split}
	&\min \{C_1, 1\} \left[\sum_{i = 1}^m u_i^2(t) + \rho^2(t)\right] \leq \sum_{i = 1}^m C_1 u_i^2(t) + \rho^2(t)   \\
	\leq &\, \sum_{i = 1}^m \left(C_1 u_i^2(0) + \rho^2(0)\right) E_\ap (- \de\, t^\ap) + \, \frac{C_2(W_0)}{\de}\,\G(\ap) \left(1 - E_\ap (- \de\, t^\ap) \right)      \\
	\leq &\, \sum_{i = 1}^m \left(C_1 u_i^2(0) + \rho^2(0)\right) \left( \frac{\G(1 + \ap)}{\G(1 + \ap) + \de\, t^\ap} \right) + \, \frac{C_2(W_0)}{\de}\,\G(\ap) ,  \quad \text{for} \; \; t \in I_{max}.
	\end{split}
\eeq
The inequalities \eqref{wij} for network weights and \eqref{Gur} for network neurons together show that the condition \eqref{dp} in Lemma \ref{L3} is satisfied by all the solutions $g(t; g^0)$ of the IVP \eqref{pb} for any initial state $g^0$. Therefore, the existence time interval  $I_{max} = [0, \infty)$ for every local solution $g(t;g^0) = \text{col}\, (u_i (t), w_{ij}(t), \rho(t))$ of this IVP \eqref{pb}. For any given initial state there exists a unique global solution of this neural network $\mathcal{NW}$. The proof is completed.
\end{proof}

The next theorem exhibits dissipative dynamics of this fractional neural network.

\begin{theorem} \label{T2}
	For fractional memristive Hopfield neural networks represented by the model \eqref{Meq} with the Assumption \eqref{Asp}, there exists a bounded absorbing set in the space $\mathbb{R}^{m + 1}$ of neuron variables $(u_1, \cdots, u_m, \rho)$, which is the ball
\beq \label{Br}
	B_\ap^* = \{ (u_1, \cdots, u_m, \rho) \in \mathbb{R}^{m +1}: \|(u_1, \cdots, u_m, \rho) \|^2 \leq M(\ap)\},
\eeq 
where 
\beq \bl{M}
	M(\ap) = 1 +  \frac{m \G(\ap) \left(\frac{m \ga^2}{b} + b \right) \left[J +  \gb \left[1 + \frac{1}{c}\sqrt{m \G(\ap)} \gl \gb^2 \right]\,\right]^2}{\delta \min \left\{\left(\frac{m \ga^2}{b} + b \right), \left(a - \frac{1}{2} k \eta^2\right)\right\} \left(a - \frac{1}{2} k \eta^2\right)}
\eeq
is a uniform positive constant independent of any initial state.
\end{theorem}

\begin{proof}
This is the consequence of the uniform estimate \eqref{cur}-\eqref{C1} for all the solutions shown in the proof of Theorem \ref{T1}. From \eqref{Gur} and \eqref{C2} we get
\begin{equation}  \bl{Mb}
	\begin{split}
	&\min \{C_1, 1\} \left[\sum_{i = 1}^m u_i^2(t) + \rho^2(t)\right] \leq \sum_{i = 1}^m \left(C_1 u_i^2(0) + \rho^2(0)\right) \left( \frac{\G(1 + \ap)}{\G(1 + \ap) + \de\, t^\ap} \right)  \\
	+ &\,\frac{1}{\de}\, m\G(\ap) \left(\frac{m \ga^2}{b} + b \right) \left[J +  \gb \left[W_0\sqrt{G_\ap(t)} + \frac{1}{c}\sqrt{m \G(\ap)} \gl \gb^2 \right]\,\right]^2 \left(a - \frac{1}{2} k \eta^2\right)^{- 2}
	\end{split}
\end{equation}
for $t > 0$. Note that
$$
	\lim_{t \to \infty} G_\ap(t) = \lim_{t \to \infty} \frac{\G(1 + \ap)}{\G(1 + \ap) + c\, t^\ap} = 0,  \qquad  \lim_{t \to \infty} \frac{\G(1 + \ap)}{\G(1 + \ap) + \delta\, t^\ap} = 0.
$$
The summation term on the right-hand side and the term $W_0 \sqrt{G_\ap(t)}$ in \eqref{Mb} converge to zero $(< 1)$  as $t \to \infty$. Divide \eqref{Mb} by $\min \{C_1, 1\} = \min \left\{\frac{ \frac{m\ga^2}{b} + b}{a - \frac{1}{2} k \eta^2},1\right\}$. Then we obtain 
\beq  \bl{sb}
	\limsup_{t \to \infty} \left( \sum_{i = 1}^m u_i^2(t) + \rho^2(t) \right) < M(\ap).
\eeq 
Moreover, it is easy to see that the limsup convergence is uniform for all initial states of the neurons confined in any given bounded set of $\mathbb{R}^{m+1}$. By Definition \ref{Ab}, the ball $B_\ap^*$ described in \eqref{Br}-\eqref{M} is an absorbing set so that the fractional memristive Hopfield neural network $\mathcal{NW}$ is a globally dissipative dynamical system.
\end{proof}

\section{\textbf{Main Result on Robust Synchronization}} 

In this section we shall prove the main result on robust synchronization of the fractional memristive Hopfield-Hebbian neural networks. The approach is to tackle the fractional differencing equations of the pairwise gap functions of the neuron state variables between any two nodes in the network. Then conduct \emph{a priori} estimates to sharply derive a computable threshold condition in terms of the known parameters, which will trigger occurrence of robust synchronization of the entire neural network once this condition is satisfied by the interneuron coupling strength. 

The gap functions between any two neuron nodes $\mathcal{N}_i$ and $\mathcal{N}_j$ are denoted by 
\begin{gather*}
	U_{ij} (t) = u_i(t) - u_j (t), \quad \text{for} \;\;  1\leq i \neq j \leq m. 
\end{gather*}
These gap functions satisfy the fractional differencing equations 
\beq \bl{deq}
	\begin{split}
	\dca U_{ij}(t) =&\, - \left(a_i U_{ij} + (a_i - a_j) u_j\right) + \sum_{\ell =1}^m \, \left(w_{i\ell} - w_{j\ell}\right) f_\ell (u_\ell)  \\
	&\,+ k\, (\psi_i (\rho) u_i - \psi_j (\rho) u_j) + (J_i - J_j) - m PU_{ij},
	\end{split}
\eeq
where the last term is $- P \sum_{\ell =1}^m [(u_i - u_\ell) - (u_j - u_\ell)] = - P\, \sum_{\ell = 1}^m (u_i - u_j) = - mP U_{ij}$.

We introduce the following secondary parameters involved in the estimates of the fractional differencing equations \eqref{deq} of the gap functions:
\begin{gather*}
	a^* = \max \{|a_i - a_j|: 1 \leq i, j \leq m \},   \\[2pt]
	\eta^* = \max \{|\eta_i - \eta_j|: 1 \leq i, j \leq m\},  \quad   J^* = \max \{|J_i - J_j|: 1 \leq i, j \leq m\}.
\end{gather*}
The main result of this work is the following theorem on the threshold condition for the interneuron coupling strength to achieve a robust synchronization of the entire fractional neural network $\mathcal{NW}$ described in \eqref{Meq}. The methodology in the proof is based on the global dissipative dynamics of this neural network and through analytic \emph{a priori} estimation of the fractional differencing equations. 

\begin{theorem} \bl{TM}
	The time-fractional memristive Hopfield-Hebbian neural network $\mathcal{NW}$ presented in the model \eqref{Meq} with the Assumption \eqref{Asp} is robust synchronizable. For any prescribed small gap $\ve > 0$, there exists a constant threshold 
\beq \bl{Ph}
	P_\ap^*(\ve) = \frac{1}{m\, \ve} \left[a^* M^{1/2}(\ap) + k \eta^* M^{3/2}(\ap) + 2\gb \left[1 + \frac{1}{c}\sqrt{m \G(\ap)} \gl \gb^2 \right] + J^*\right]  > 0,
\eeq 
where $M(\ap)$ is shown in \eqref{M}, such that if the interneuron coupling strength coefficient $P > P_\ap^*(\ve)$, then the neural network $\mathcal{NW}$ is robust synchronized. Namely, for any prescribed small gap $\ve > 0$,
\beq \bl{asd}
	\deg_s \,(\mathcal{NW}) = \sup_{(u_1^0, \,\cdots, \,u_m^0, \,\rho^0)\, \in \,\mathbb{R}^{m +1}} \left\{ \max_{1 \leq i < j \leq m} \left\{\limsup_{t \to \infty} \, |u_i (t) - u_j(t)| \right\}\right\} < \ve.
\eeq
Moreover, the synchronization threshold $P_\ap^* (\ve)$ is strictly decreasing with respect to the fractional-order $\ap \in (0, 1)$.  
\end{theorem}

\begin{proof}
In the sequel of this proof we can denote $U_{ij}(t)$ by $U(t)$ for notational simplicity. For any given $1 \leq i \neq j \leq m$, multiply the $U_{ij}$-equation \eqref{deq} by $U_{ij} (t)$. With \eqref{wij} and \eqref{sb} we have  

\begin{equation} \bl{Gb}
	\begin{split}
	&\frac{1}{2}\,\dca U^2 (t) \leq U(t) \dca U(t) = - a_i U^2 (t) - (a_i - a_j) u_j (t) U(t)   \\
	&+ \sum_{\ell =1}^m \,\left(w_{i\ell}(t) - w_{j\ell}(t) \right) f_\ell (u_\ell (t))U(t)   \\[4pt]
	&+ k (\psi_i (\rho(t)) - \psi_j (\rho(t)))u_i(t) U(t) + k \psi_j (\rho(t)) U^2(t) + (J_i - J_j) U(t) - mPU^2 (t)  \\[3pt]
	\leq & - a U^2(t) + a^* |u_j(t)| |U(t)| + 2 \gb \left[W_0 \sqrt{G_\ap(t)} + \frac{1}{c}\sqrt{m \G(\ap)} \gl \gb^2 \right] |U(t)|   \\[7pt]
	&\, + k |\eta_i - \eta_j| \rho(t) |u_i(t)| |U(t)| + k \left(\eta_j \rho(t) - \rho^2 (t) \right) U^2(t) + J^* |U(t)| - m PU^2 (t)  \\[2pt]
	\leq & - a U^2(t) + a^* |u_j(t)| |U(t)| +  2 \gb \left[W_0 \sqrt{G_\ap(t)} + \frac{1}{c}\sqrt{m \G(\ap)} \gl \gb^2 \right] |U(t)|   \\[2pt]
	&\, + k \eta^* |\rho(t)| |u_i(t)| |U(t)| + \frac{1}{2}\, k \,(\eta^2 - \rho^2 (t)) U^2(t) + J^* |U(t)| - m PU^2 (t)  \\
	\leq &\, - \left[ a - \frac{1}{2} k \,\eta^2 + mP \right] U^2(t)   \\
	&\, + \left[ a^* |u_j(t)| + k \eta^* |\rho(t)| |u_i(t)| +  2\gb \left[W_0 \sqrt{G_\ap(t)} + \frac{1}{c}\sqrt{m \G(\ap)} \gl \gb^2 \right] + J^* \right] |U(t)|.
	\end{split}
\end{equation}
According to the ultimate bounding estimate \eqref{sb} and \eqref{Gt}, for any given initial state and weights $g^0 = (u_i^0, w_{ij}^0, \rho^0)$, there is a finite time $T(g^0) > 0$ such that 
$$
	 \sum_{i = 1}^m u_i^2(t) + \rho^2(t) < M(\ap) \quad \text{and} \quad  W_0 \sqrt{G_\ap(t)} < 1, \quad \text{for} \;\; t > T(g^0).
$$
Therefore the above differential inequality \eqref{Gb} implies that, for $t > T(g^0)$,
\begin{equation} \bl{Db}
	\begin{split}
	& \dca U^2 (t) + 2 \left(a - \frac{1}{2}\, k \eta^2 + mP \right)U^2(t)    \\
	\leq &\, 2 \left(a^* M^{1/2}(\ap) + k \eta^* M^{3/2}(\ap) + 2\gb \left[1 + \frac{1}{c}\sqrt{m \G(\ap)} \gl \gb^2 \right]  + J^* \right) |U(t)|.
	\end{split}
\end{equation}
Use Cauchy inequality to split the product on the right-hand side of the inequality \eqref{Db} into two terms. Then we obtain 

\beq \bl{Ub}
	\begin{split}
	 &\dca U^2 (t) +  \left(a - \frac{1}{2}\, k \eta^2 + mP \right)U^2(t)    \\
	\leq &\, \frac{\left(a^* M^{1/2}(\ap) + k \eta^* M^{3/2}(\ap) + 2\gb \left[1 + \frac{1}{c}\sqrt{m \G(\ap)} \gl \gb^2 \right]  + J^* \right)^2}{a - \frac{1}{2}\, k \eta^2 + mP}, \quad t > T(g^0).
	\end{split}
\eeq
Denote by 
\beq \bl{aph}
	\mu = a - \frac{1}{2}\, k \eta^2 + mP.
\eeq
From \eqref{Ub} it follows that 
\beq  \bl{UG}
	\dca U^2 (t) + \mu U^2(t) \leq \frac{1}{\mu} \left(a^* M^{1/2}(\ap) + k \eta^* M^{3/2}(\ap) + 2\gb \left[1 + \frac{1}{c}\sqrt{m \G(\ap)} \gl \gb^2 \right]  + J^* \right)^2
\eeq
for $ t > T(g^0)$. Apply the fractional Gronwall inequality \eqref{pq} and \eqref{fG} in Lemma \ref{L2} to the fractional differential inequality \eqref{UG} on the time interval $[T(g^0), \infty)$. Then \eqref{UG} shows that for any two different neuron nodes $\mathcal{N}_i$ and $\mathcal{N}_j$ in the neural network, the gap function $U_{ij}(t) (= U(t))$ of their state trajectories satisfies
\beq \bl{Vb} 
	\begin{split}
	&U_{ij}^2(t) \leq U_{ij}^2 (T_{g^0}) \left( \frac{\G(1 + \ap)}{\G(1 + \ap) + \mu (t - T(g^0))^\ap} \right)  \\
	+ &\, \frac{\G(\ap)}{\mu^2} \left[a^* M^{1/2}(\ap) + k \eta^* M^{3/2}(\ap) + 2\gb \left[1 + \frac{1}{c}\sqrt{m \G(\ap)} \gl \gb^2 \right] + J^*\right]^2,  \; t  > T(g^0).
	\end{split}
\eeq  

Finally, for any prescribed small gap $\ve > 0$, if the threshold condition $P > P_\ap^*(\ve)$ in \eqref{Ph} is satisfied, then by \eqref{Vb} and because $a - \frac{1}{2}k \eta^2 > 0$ due to \eqref{Asp} we conclude that for any $1 \leq i \neq j \leq m$,
\beq \bl{Vsp}
	\begin{split}
	&\limsup_{t \to \infty} |u_i(t) - u_j(t)| = \limsup_{t \to \infty} |U_{ij}(t)|   \\[3pt]
	 \leq &\,\frac{\sqrt{\G(\ap)}}{\mu} \left[a^* M^{1/2}(\ap) + k \eta^* M^{3/2}(\ap) + 2\gb \left[1 + \frac{1}{c}\sqrt{m \G(\ap)} \gl \gb^2 \right] + J^*\right]   \\
	 = &\, \frac{1}{a - \frac{1}{2}\, k \eta^2 + m P} \left[a^* M^{1/2}(\ap) + k \eta^* M^{3/2}(\ap) + 2\gb \left[1 + \frac{1}{c}\sqrt{m \G(\ap)} \gl \gb^2 \right] + J^*\right] \\
	 < &\, \frac{1}{m P_\ap^*(\ve)} \left[a^* M^{1/2}(\ap) + k \eta^* M^{3/2}(\ap) + 2\gb \left[1 + \frac{1}{c}\sqrt{m \G(\ap)} \gl \gb^2 \right] + J^*\right] = \ve.
	\end{split}
\eeq
Thus \eqref{asd} is proved. Since $\G^\prime (\ap) < 0$ for $\ap \in (0, 1)$, $\G(\ap)$ is strictly decreasing in this interval. Therefore, from \eqref{Ph} and \eqref{M} it is seen that the synchronization threshold $P_\ap^* (\ve)$ is strictly decreasing with respect to the fractional order $\ap \in (0, 1)$. The proof is completed.
\end{proof}

Remark. For fractional order $0 < \ap < 1$, the Mittag-Leffler function $E_\ap(- \mu t^\ap)$ as $t \to \infty$ does not have an exponential decay rate. According to \eqref{pq}-\eqref{fG} and \eqref{Vb}, we can see that the robust synchronization in Theorem \ref{TM} admits a slower fractional power convergence rate $O([\mu\, t^\ap]^{-1})$. This is different from the approximate exponential synchronization rate of the ODE modeled memristive Hopfield neural networks in the earlier paper \cite{Y}. 

\begin{corollary} \bl{Cy}
	Any recursive fractional memristive Hopfield neural network presented in the model \eqref{Meq} with the Assumption \eqref{Asp} except that all learning weights $w_{ij}, 1 \leq i, j \leq m,$ are fixed constants is robust synchronizable. For any prescribed small gap $\ve > 0$, there is a constant threshold 
\beq \bl{Pw}
	\Hat{P}_\ap (\ve) = \frac{1}{m\, \ve} \left[a^* L^{1/2}(\ap) + k \eta^* L^{3/2}(\ap) + 2\gb W + J^*\right]  > 0,
\eeq 
where $W= \max \{|w_{ij}|:  \leq i, j \leq m\}$ and
\beq \bl{L}
	L(\ap) = 1 + \frac{m \G(\ap) \left(\frac{m \ga^2}{b} + b \right) (\gb W + J)^2}{\delta \min \left\{\left(\frac{m \ga^2}{b} + b \right), \left(a - \frac{1}{2} k \eta^2\right)\right\} \left(a - \frac{1}{2} k \eta^2\right)},
\eeq
such that if the interneuron coupling strength coefficient $P > \Hat{P}_\ap (\ve)$, then the neural network is robust synchronized. 
\end{corollary}

\begin{proof}
	Since all the weights $\{w_{ij}\}$ in the neural network are fixed constants, we should let $W= \max \{|w_{ij}|:  \leq i, j \leq m\}$ replace $W_0 \sqrt{G_\ap(t)} + \frac{1}{c}\sqrt{m \G(\ap)} \gl \gb^2$ in \eqref{ut}, \eqref{cur}, \eqref{C2} and later replace $1 + \frac{1}{c}\sqrt{m \G(\ap)} \gl \gb^2$ in \eqref{M}, \eqref{Ph}. 
	
	Then we reach the same conclusion that any such recursive fractional memristive Hopfield neural network will be robust synchronized if the threshold condition $P > \Hat{P}_\ap (\ve)$ shown in \eqref{Pw}-\eqref{L} is satisfied by the network coupling coefficient. 
\end{proof}

In particular, by the strict decreasing property of $\G(\ap)$ on the interval $(0, 1)$ and the known value $\G(0.5) = \sqrt{\pi}$, we have
$$
	L(\ap) = 1 + \frac{m \sqrt{\pi} \left(\frac{m \ga^2}{b} + b \right) (\gb W + J)^2}{\delta \min \left\{\left(\frac{m \ga^2}{b} + b \right), \left(a - \frac{1}{2} k \eta^2\right)\right\} \left(a - \frac{1}{2} k \eta^2\right)}, \quad \text{if} \;\; 0.5 \leq \ap < 1.
$$

Although Hopfield neural networks are cross-coupled networks but not multilayer or deep neural networks, the procedures and techniques in this work can be applied to study of interlayer and intralayer synchronization of multilayer learning neural networks as well as the cluster synchronization of many multiplex networks. 

\vspace{4pt}
\textbf{Conclusions}

\vspace{4pt}
We summarize the contributions in this paper and outlook for further extensions.

First we proposed and investigated a new model.of Caputo type time-fractional memristive Hopfield-Hebbian neural networks. This model incorporates multiple features of long-term memory, efficient stinulus responses in brain-machine learning, and Hebbian learning rules in unsupervised meaning.  

The fractional analysis of three Lemmas are presented and proved as a foundation for conducting the scaled \emph{a priori} estimates of solutions, especially the fractional Gronwall inequality not just in the form of Mittag-Lefflrter functions but also in the explicit form in terms of Gamma functions of a fractional order, which enables us to effectively establish the global solution existence and the ultimate absorbing bound. It paves the way to explore the asymptotic dynamics of the time-fractional neural networks.

The main contribution of this paper is Theorem \ref{TM}. It rigorously proved a sufficient threshold condition to be satisfied by a single interneuron coupling strength coefficient for achieving robust synchronization of the Caputo time-fractional neural networks. The usefulness and merit of this result lie in the explicit expression of the threshold and the convergence rate, which are computable in terms of the original parameters. Moreover, threshold is shown to be decreasing when the fractional order $\ap \in (0, 1)$. 

The methodology and analysis in this work potentially can be generalized to study other types of time-fractional artificial neural networks, especially for multilayer forward neural networks with semi-supervised data sets and for multiplex networks. 

Mathematical research on predicable or complex asymptotic dynamics including  synchronization of artificial neural networks is meaningful in deep learning performances but remains challenging and widely open, especially for training and learning with backpropagation algorithms.

\end{document}